\documentclass[11pt,a4paper]{amsart}

\usepackage[a4paper]{geometry}
\usepackage{amsmath, amssymb, amsthm, amsfonts, mathtools}
\usepackage{amsrefs}
\usepackage[usenames,dvipsnames]{xcolor}
\usepackage{mdframed}
\usepackage[colorlinks=true,linkcolor=Red,citecolor=Green]{hyperref}
\usepackage{bbm}
\usepackage{enumerate}

\usepackage{tikz}

\newcommand{\NN}{\mathbb N}
\newcommand{\RR}{\mathbb R}
\newcommand{\CC}{\mathbb C}

\newcommand{\DtN}{Dirichlet-to-Neumann }

\newcommand{\ep}{\epsilon}
\newcommand{\pa}{\partial}

\newcommand{\scal}[2]{\langle #1, #2 \rangle}

\newcommand{\dd}{\mathrm{d}} 
\newcommand{\id}{\mathbbm{1}}
\DeclareMathOperator{\ddiv}{div}
\DeclareMathOperator{\sspan}{span}
\DeclareMathOperator{\diag}{diag}
\DeclareMathOperator{\supp}{supp}

\newcommand{\normal}{\mathbf{n}}

\newtheorem{lemma}{Lemma}[section]
\newtheorem{theorem}[lemma]{Theorem}
\newtheorem{proposition}[lemma]{Proposition}
\newtheorem{coro}[lemma]{Corollary}
\newtheorem{assumption}{Assumption}
\newtheorem*{assumption*}{Assumption}

\theoremstyle{definition}

\theoremstyle{remark}
\newtheorem{remark}[lemma]{Remark}
\newmdtheoremenv{rem}[lemma]{Remark}

\let\Im\undefined
\let\Re\undefined

\DeclareMathOperator{\Re}{Re}
\DeclareMathOperator{\Im}{Im}
\newcommand{\Ctwo}{\tilde{C}^{(2)}}
\newcommand{\Cone}{\tilde{C}^{(1)}}

\title{A Partial Data Problem in Linear Elasticity}

\author[M. Doll]{Moritz Doll}
\address{Department 3 -- Mathematics, University of Bremen, Bibliotheksstr. 5, \newline\indent D-28359 Bremen, Germany}
\email{doll[AT]uni-bremen.de}

\author[A. Froehly]{Andr\'e Froehly}
\address{Institut f\"ur Analysis, Leibniz Universit\"at Hannover, Welfengarten 1, \newline\indent D-30167 Hannover, Germany}
\email{andre.haenel[AT]math.uni-hannover.de}

\author[R. Schulz]{Ren\'e Schulz}
\address{Institut f\"ur Analysis, Leibniz Universit\"at Hannover, Welfengarten 1, \newline\indent D-30167 Hannover, Germany}
\email{rschulz[AT]math.uni-hannover.de}

\thanks{We would like to thank Gunther Uhlmann for helpful suggestions and Masaru Ikehata for clarifying the history of the reduced Lam\'e system.}

\begin{document}
\begin{abstract}
	We discuss the determination of the Lam\'e parameters of an elastic material by the means of boundary measurements.
    We will combine previous results of Eskin--Ralston and Isakov to prove inverse results in the case of bounded domains with partial data.
    Moreover, we generalise these results to infinite cylinders. 
\end{abstract}
\maketitle
\section{Introduction}

Given an inhomogeneous, isotropic elastic body, we want to determine the Lam\'e parameters $\lambda$ and $\mu$ by measuring the ``deformation-to-stress'' map $\Lambda$ on the boundary.
Mathematically, $\Lambda$ is the \DtN operator.

Let $\Omega \subseteq \RR^2 \times \RR_{>0}$ be a domain with smooth boundary $\Gamma \coloneqq \partial \Omega$ and
set $\Gamma_1 \coloneqq \Gamma \cap (\RR^2 \times \RR_{>0})$ and $\Gamma_0 = \Gamma \setminus \overline{\Gamma_1}$.
We assume that $\Gamma_0$ has Lipschitz boundary.
\begin{center}
\begin{figure}[ht!]
\label{fig:setup-bdd}
\begin{tikzpicture}

    \filldraw[fill=lightgray, draw=black] plot [smooth cycle] coordinates {(-0.5,-0.3) (0.5,-0.3) (1,0) (0.7,1) (0.5,2) (-1.1,2) (-1,0)};
    \filldraw[fill=white,draw=white] (-0.5,-0.3) rectangle (0.5,-0.4);

	\draw[very thick] (-0.5,-0.3) -- (0.5,-0.3) node [below, midway] {\textcolor{black}{$\Gamma_0$}};

	\node at (-0.3,1) {$\Omega$};
    \node at (1.1,1.2) {$\Gamma_1$};
\end{tikzpicture}
\caption{Bounded domain}
\end{figure}
\end{center}
We consider two cases: bounded domains and infinite cylinders.
An unbounded domain $\Omega$ is called an infinite cylinder if 
there exists a set $G \subseteq \RR^2$ with smooth boundary $\partial G$ and $R> 0$ such that
\begin{align*}
	\Omega \cap (\RR^3 \setminus [-R,R]^3) = \RR \times G . 
\end{align*}
Moreover, for the sake of simplicity, we additionally assume that $\Gamma_1 = \partial \Omega \cap (\RR^2 \times \{0\})$ is bounded.

The basic assumptions on the Lam\'e parameters $\mu_j , \lambda_j \in  C^\infty(\overline{\Omega})$, $j=1,2$ are
\begin{align}
\label{eq:lameass}
	\mu_j (x) , 3\lambda_j(x) + 2 \mu_j(x)  > 0 \text{ for all } x \in \overline{\Omega} . 
\end{align}
Since the proof of the main result will be based on a reflection argument, we also assume the following:
\begin{assumption}
    There exists an open set $U \subseteq \RR^3$ with $\overline{\Omega} \subseteq U$ and functions $\tilde\lambda_j, \tilde\mu_j \in C^\infty(U)$ such that
    $\tilde\lambda_j|_{\overline\Omega} = \lambda_j$ and $\tilde\mu_j|_{\overline\Omega} = \mu_j$ with the property that
    \[\partial_{x_3}^k \tilde\mu_j(x) = 0,\quad \partial_{x_3}^k \tilde\lambda_j(x) =0 \quad\text{ for all odd } k \in \NN\] and $x\in U \cap (\RR^2 \times \{0\})$.
\end{assumption}%

For the infinite cylinder, we add the following restriction on the Lam\'e parameters:
\begin{assumption}
    There exist a compact set $K \subseteq \overline{\Omega}$ such that
    $\lambda_1(x) = \lambda_2(x)$ and $\mu_1(x) = \mu_2(x)$ for $x \in \overline{\Omega} \setminus K$, and $\lambda_j$ and $\mu_j$ are independent of $x_3$ on $\overline{\Omega} \setminus K$.
\end{assumption}
\begin{center}
\begin{figure}[ht!]
\label{fig:setup}
\begin{tikzpicture}

	\fill[lightgray, opacity=1/8] (-3,0) rectangle (-1.5,2);
	\fill[lightgray, opacity=1/8] (1.5,0) rectangle (3,2);
    \fill[lightgray, opacity=1/2] (-1.5,-0.5) rectangle (1.5,2.5);
    \draw[fill=lightgray, draw=gray, domain=-5.9:5.9, samples=50] (1,0) to[bend right] (0.8,2) -- plot ({-0.15*\x-0.1}, {2 + 0.2/cosh(\x - 1) - 0.2/cosh(\x + 1)}) -- (-1,2) to[bend right] (-1,0) to[bend left] (-0.8,-0.13) to[bend right] (-0.5,-0.3) -- (0.5,-0.3) to[bend right] (0.8,-0.13) to[bend left] (1,0);

	\draw (-3,2)--(-1,2);
	\draw (0.8,2)--(3,2);
    \draw (-3,0)--(-1,0);
    \draw (1,0) -- (3,0) node [below right] {$\Gamma$};
    \draw[domain=-6:6,samples=50] plot ({-0.15*\x-0.1}, {2 + 0.2/cosh(\x - 1) - 0.2/cosh(\x + 1)});
    \draw (-1,0) to[bend left] (-0.8,-0.13) to[bend right] (-0.5,-0.3); 
    \draw (1,0) to[bend right] (0.8,-0.13) to[bend left] (0.5,-0.3); 
	\draw[very thick] (-0.5,-0.3) -- (0.5,-0.3) node [below left, midway] {\textcolor{black}{$\Gamma_0$}};

	\node at (-0,1) {$K$};
	\node at (2.5,1) {$\Omega$};
    \node at (2.3,2.3) {$[-R,R]^3$};
\end{tikzpicture}
\caption{Infinite cylinder}
\end{figure}
\end{center}

We consider the \DtN operator $\Lambda^{(j)}$ for Lam\'e parameters $\lambda_j,\mu_j$ with simply supported boundary conditions\footnote{We could also choose soft clamped boundary conditions, see Remark~\ref{rem:clamped} below.} on $\Gamma_0$.
\begin{theorem}\label{thm:main1}
	 Let  $\Lambda^{(1)} = \Lambda^{(2)}$ and  $\|f\|^2_\tau \coloneqq 
	\int_{\Omega} e^{2\tau|x|^2} |f(x)|^2 dx.$ Then for $\tau \gg 0$, we have 
    	\begin{align*}
        	\|\lambda_1 +\mu_1 - (\lambda_2 + \mu_2) \|_\tau\le  C \tau^{-1} \|\mu_1 - \mu_2\|_\tau ,
    	\end{align*}
    	for a constant $C >0$ independent of $\tau$.
\end{theorem}
Moreover, we have  the following result, which corresponds to  \cite{EsRa}*{Theorem 3}.
\begin{theorem}\label{thm:main2}
    Let  $b_j = \mu_j (\lambda_j + \mu_j)(2\lambda_j + 4 \mu_j)^{-1}$, $j=1,2$, and 
    \begin{align*}
        Z \coloneqq \{ \theta \in \CC^3\setminus\{0\} \colon \theta \cdot \theta = 0\}.
    \end{align*}
    If $\Lambda^{(1)} = \Lambda^{(2)}$, then we have for all $\theta \in Z$ that
    \begin{align*}
        b_1^{1/2} (\theta \cdot \partial_x)^2 (b_1^{-1/2}) -  b_1^{1/2} (\theta \cdot \partial_x)^2 (\mu_1^{-1}) =    b_2^{1/2} (\theta \cdot \partial_x)^2 (b_2^{-1/2}) -  b_2^{1/2} (\theta \cdot \partial_x)^2 (\mu_2^{-1}) . 
    \end{align*}
\end{theorem}
Note that Theorem~\ref{thm:main1} implies $(\mu_1,\lambda_1) = (\mu_2,\lambda_2)$ if we already know that either $\mu_1 = \mu_2$ or $\lambda_1 = \lambda_2$.
In the case of the full data problem and bounded domains, Theorem~\ref{thm:main1} and Theorem~\ref{thm:main2} are due to Eskin--Ralston~\cite{EsRa}.

In the case of the infinite cylinder, we additionally obtain a new result for the full data problem, that is $\Gamma_0 = \varnothing$:
\begin{theorem}\label{thm:inf-cyl}
	There exists a constant $\varepsilon > 0$ (depending only on $K$) such that, if $\| \nabla \mu_i \|_{C^2(K)} < \varepsilon$, then $\Lambda^{(1)} = \Lambda^{(2)}$ implies that $(\mu_1, \lambda_1) = (\mu_2, \lambda_2)$.
\end{theorem}

\subsection*{History}
This problem is related to the inverse problem of electrical impedence tomography (EIT), where one considers the equation $\nabla \cdot (q\nabla u) = 0$ with bounded and positive potential $q$ on a bounded domain $\Omega$ with sufficiently regular boundary $\Gamma = \partial \Omega$.
Here, the \DtN operator $\Lambda_q$ is viewed as the ``voltage-to-current'' map.

In the pioneering works of Calder\'on~\cite{Calderon} and Sylvester--Uhlmann~\cite{SylvesterUhl}, the authors constructed so-called complex geometric optics solutions (we refer to the surveys \cites{UhlmannDev,UhlmannUnseen} for more details).
If two potentials $q_1, q_2$ satisfy $\Lambda_{q_1} = \Lambda_{q_2}$ on the boundary, then it was shown in \cite{SylvesterUhl} that $q_1 = q_2$.
There are also results for unbounded domains, for instance the slab was treated by Li--Uhlmann~\cite{LiUhl}.

In applications, it is usually not possible to measure $\Lambda$ on the whole boundary. Therefore, it is natural to ask whether it suffices to know equality of $\Lambda_{q_1}$ and $\Lambda_{q_2}$ on
a small part $\Gamma_0$ of the boundary $\Gamma$.
Kenig--Sj\"ostrand--Uhlmann~\cite{KSjU} showed that for the Schr\"odinger operator this is true if the \DtN operator is measured in $\Gamma_0$ and $\Gamma$ is ``convex enough'' with respect to $\Gamma_0$ (again we refer to \cite{UhlmannUnseen} and references therein).
This assumption can be relaxed if one strenghens the assumptions on the $\Gamma_0$: Isakov~\cite{Isakov} proved that if $\Gamma_0$ is part of a plane, then a reflection argument proves uniqueness.

In the case of the elasticity operator $L_j$ with Lam\'e parameters $\lambda_j$ and $\mu_j$, the problem is considerably more involved and even global uniqueness with full data is still open.
Under an a-priori smallness assumption on the parameters $\mu_j$, the inverse problem was solved by Nakamura--Uhlmann~\cites{NakUhl,NakUhl-err,NakUhl02} and Eskin--Ralston~\cite{EsRa}.
The principal symbol of the elasticity operator is not equal to the principal symbol of the matrix Laplacian, thus one cannot directly apply the method of complex geometric optics solutions.
In \cite{EsRa}, the authors considered an auxiliary equation, from which solutions of the elasticity system could be deduced (see also \cite{NakUhl02}).
To construct complex geometric optics solutions for the auxiliary equation one has to solve a $\bar\partial$-problem.

For the partial data problem for the $3$-dimensional elasticity-operator, we are only aware of the results of Imanuvilov--Uhlmann--Yamamoto~\cite{IUY}, which operated under the assumption that $\mu_1$ and $\mu_2$ are constant and
$\lambda_1 = \lambda_2$ on $\Gamma_0$.

The paper is structured as follows:
We recall the definition of the Lam\'e system and define the \DtN operator for bounded domains in Section~\ref{sec:dtn}.
Section~\ref{sec:construction} is concerned with the construction of complex geometric optics solutions.
In Section~\ref{sec:proof-bdd} we parametrize the complex geometric optics solutions to prove the main results in the case of bounded domains.
We prove the main theorems for the case of the infinite cylinder in Section~\ref{sec:cylinder} by reducing them to the previous case of bounded domains.
The paper ends with an appendix on the appropiate radiation conditions for the elasticity operator on the infinite cylinder.

\section{The Lam\'e System and the \DtN Operator}\label{sec:dtn}
We denote by $u : \Omega \to \CC^3$ the displacement field of the elastic material. The strain tensor of $u$ is given by the matrix 
\begin{equation}\label{def:strain}
  	\varepsilon(u) \coloneqq \frac{1}{2} ( \partial_j u_k + \partial_k u_j)_{j,k} = \frac12 \left( \nabla u + (\nabla u )^T \right), 
\end{equation}
where $\nabla u$ denotes the Jacobian of $u$, and the stress is defined by the following expression
\begin{equation}\label{def:stress}
  	\sigma^{(i)}(u) \coloneqq 2 \mu_i   \varepsilon(u)  + \lambda_i \ddiv (u) \id_3 , \quad i = 1,2 . 
\end{equation}
Here, $\id_3$ is the three-dimensional unit matrix. Then the corresponding Lam\'e operator reads as 
\begin{equation} 
	\label{eq:defLi}
	L^{(i)}  u \coloneqq  \nabla \cdot   \sigma^{(i)}(u)   , 
\end{equation}
where we put $(\nabla \cdot A)_j = \sum_k \pa_k A_{jk}$ for a matrix-valued function $A =(A_{jk})_{jk} \in  C^1(\Omega, \CC^{3 \times 3})$.

To obtain a well-posed problem, we have to impose suitable boundary conditions.
The outward boundary forces are given by 
\[ B^{(i)} u \coloneqq \sigma^{(i)}(u) \normal , \]
where $\normal$ is the outward unit normal vector at the boundary $\Gamma$.  Moreover, we define by
\begin{align*}
	B^{(i)}_{n} u \coloneqq \scal{\normal}{ B^{(i)} u}  \normal   , \qquad B^{(i)}_{\tau } u \coloneqq B^{(i)} u - B^{(i)}_{n} u ,
\end{align*}
the corresponding projection in the normal direction respectively  tangential direction. Likewise we define for the  displacement field $u$: 
\begin{align*}
	u_{n} \coloneqq \scal{\normal}{u} \normal   , \qquad u_{\tau } \coloneqq u - u_{ n}  . 
\end{align*}
Then we will distinguish between different operators on the boundary: in addition to   $u = (u_n, u_\tau)^T$ and $B^{(i)} u = (B^{(i)}_n u, B^{(i)}_\tau u)^T$,  we define
\begin{align*}
	C^{(i)} u \coloneqq \begin{pmatrix} B^{(i)}_n u \\  u_\tau  \end{pmatrix}. 
\end{align*}
The boundary condition $C^{(i)}u = 0$ is the \emph{simply supported boundary condition}.
\begin{remark}\label{rem:clamped}
    The boundary conditions
    \begin{align*}
        \qquad D^{(i)} u \coloneqq \begin{pmatrix} u_n \\ B^{(i)}_\tau u \end{pmatrix} = 0
    \end{align*}
    are called \emph{soft clamped boundary conditions}. We note that our theorems are still valid if we replace $C^{(i)}u = 0$ by $D^{(i)}u = 0$ on $\Gamma_0$.
\end{remark}

We recall the definition of the \DtN operator in the case that $\Omega$ is bounded. 
The definition of the \DtN operator for infinite cylinders is given in Section~\ref{sec:cylinder}.
We consider the boundary value problem 
\begin{equation}\label{eq:bvp_Poisson_neu} \left\{ 
\begin{aligned}
	L^{(i)}  u  &= 0  \text{ in } \Omega\\
	u &= g  \text{ on } \Gamma_1\\
	C^{(i)} u &= 0  \text{ on } \Gamma_0,
\end{aligned} \right. 
\end{equation}
where assumptions on $g : \Gamma_1 \to \CC^3$ will arise later. 
Using Green's formula, we obtain
\begin{align}\label{eq:Greens_formula}
	\int_\Omega \scal{\sigma^{(i)} (u)}{\varepsilon (v) }_{\CC^3 \times \CC^3} \; \dd x  &= \int_{\Omega} L^{(i)} u\, \overline{v} \, \dd x +
	\int_{\Gamma} B^{(i)} u\, \overline{v} \; \dd s ,
\end{align}
where $\scal{\cdot}{\cdot}_{\CC^3\times\CC^3}$ denotes the standard sesquilinear product on $\CC^3$.
We observe that the equation \eqref{eq:bvp_Poisson_neu} may be rewritten as follows:
if $u$ is a solution of \eqref{eq:bvp_Poisson_neu}, then
for all $v \in H^{1}(\Omega;\CC^3)$ with $v = 0$ on $ \Gamma_1$ and $v_\tau =0$ on $\Gamma_0$, we have that
\begin{equation}\label{eq:bvp_Poisson_weak}
	\int_\Omega \scal{\sigma^{(i)} (u)}{\varepsilon (v) }_{\CC^3 \times \CC^3} \; \dd x  = 0.
\end{equation}
For the sake of completeness we will consider the well-posedness of the boundary value problem \eqref{eq:bvp_Poisson_neu} and recall the corresponding definition of the Sobolev spaces.  
Let  $T\Gamma$  be the tangent bundle of $\Gamma$. By means of local charts we may define for $s \ge 0$ the corresponding Sobolev spaces $H^s(\Gamma; T \Gamma)$ and by duality the spaces $H^{-s} (\Gamma; T \Gamma) \coloneqq H^s(\Gamma; T \Gamma)^*$. 
For a relatively open subset $U\subseteq \Gamma$ and $s \ge 0$ we put  
\begin{align}\label{def:Hs}
    H^s(U; TU) &\coloneqq \{ G|_U :  G \in   H^s(\Gamma;T\Gamma) \} ,\\ 
    H_{00}^s(U; TU) &\coloneqq \{ g \in H^s(\Gamma;T \Gamma) :  \supp(g) \subseteq \overline{U}  \} , \label{def:Hs_00}
\end{align}
as well as 
\begin{align*}
    H^{-s}(U; TU) &\coloneqq H_{00}^s(U; TU)^* , \quad \text{and} \quad H_{00}^{-s}(U; TU) \coloneqq 
    H^s(U; TU)^* .  
\end{align*}
Note that that Formulae \eqref{def:Hs} and \eqref{def:Hs_00} hold for $s < 0$, after the correct interpretation of the right-hand sides, cf. \cite{McLean}*{Theorem 3.30}.
In a similar way, we may define $H^s(U;T^\perp U)$ and $H^s_{00}(U;T^\perp U)$, where $T^\perp U$ is the normal bundle, as well as  $H^s(U;\CC^3) , H^s_{00}(U;\CC^3)$.
Note that
for arbitrary $s \in \RR$,
\begin{align*}
    H^s(U;\CC^3) &= H^s( U; T^\perp U) \oplus H^s(U;TU)  , \\
    H^s_{00} (U;\CC^3) &= H^s_{00}(U;T^\perp U) \oplus H^s_{00} ( U; T U).
\end{align*}
Considering again the boundary value problem \eqref{eq:bvp_Poisson_neu} the Fredholm property reads as follows: 
\begin{lemma}
	Let $g  \in H^{1/2} (\Gamma_1;T^\perp \Gamma_1) \oplus H^{1/2}_{00}(\Gamma_1;T\Gamma_1)$.
    Then the boundary value problem \eqref{eq:bvp_Poisson_neu} is uniquely solvable if and only if corresponding homogeneous problem is uniquely solvable.
\end{lemma}
The proof is similar to the proof of \cite{McLean}*{Theorem 4.11}.
Moreover, a simple calculation shows that \eqref{eq:bvp_Poisson_neu} with $g = 0$ is uniquely solvable. Hence,  \eqref{eq:bvp_Poisson_neu} has a unique solution $u \in H^1(\Omega;\CC^3)$.
Next we define  the corresponding \DtN operator.
We observe that if $u \in H^1(\Omega;\CC^3)$ with $L^{(i)} u = 0$ in $\Omega$, then we may define $B^{(i)} u \in H^{-1/2}(\Gamma ;\CC^3)$ by Formula \eqref{eq:Greens_formula}, i.e. we have for all $v \in H^1(\Omega;\CC^3)$ that
\begin{align*}
    \scal{B^{(i)} u}{v|_{\Gamma}} = \int_\Omega \scal{\sigma^{(i)}  (u)}{\varepsilon (v) }_{\CC^3 \times \CC^3} \; \dd x.
\end{align*}
Finally, we set  
\begin{align*}
    \Lambda^{(i)}  g  \coloneqq B^{(i)} u |_{\Gamma_1}  ,
\end{align*}
where $u \in H^1(\Omega;\CC^3)$ solves \eqref{eq:bvp_Poisson_neu}.
From the construction it is clear that the \DtN operator $\Lambda^{(i)}$ has the following mapping properties 
\[ \Lambda^{(i)} : H^{1/2} (\Gamma_1;T^\perp \Gamma_1) \oplus H^{1/2}_{00}(\Gamma_1;T\Gamma_1)  \to H^{-1/2}_{00} (\Gamma_1;T^\perp\Gamma_1) \oplus H^{-1/2}(\Gamma_1;T\Gamma_1) . \]

As a next step we consider on $H^1(\Omega,\CC^3)$ the sesquilinear form
\begin{equation*}
H_\Omega(u_1 ,u_2 )  \coloneqq \int_\Omega \{ \lambda_2 - \lambda_1 \} \ddiv (u_1) \overline{\ddiv(u_2)} 
+ 2\{\mu_2 - \mu_1\}  \scal{\varepsilon (u_1)}{\varepsilon(u_2)}_{\CC^3 \times \CC^3} \; \dd x , 
\end{equation*}
which will be crucial in what follows. In the case that the Dirichlet-to-Neumann maps are equal we obtain the following result. 
\begin{lemma}\label{lem:h-omega}
	If $\Lambda^{(1)} = \Lambda^{(2)}$, then we have $H_\Omega(u_1,u_2) = 0$ for all $u_1, u_2 \in H^1(\Omega;\CC^3)$ that satisfy
	\begin{equation*}\left\{ 
		\begin{aligned}
		L^{(i)}  u_i  &= 0  \quad\text{ in } \Omega\\
		C^{(i)} u_i &= 0  \quad\text{ on } \Gamma_0.
	\end{aligned} \right. 
	\end{equation*}
\end{lemma}
\begin{proof}
Using the previous notations we have
\begin{align*}
    H_{\Omega}(u_1,u_2) = \int_{\Omega} \scal{\sigma_2(u_1) - \sigma_1(u_1)}{\varepsilon(u_2)} \; \dd x. 
\end{align*}
Let $u_1,u_2$ be given as above. Then we choose  $v \in H^1(\Omega;\CC^3)$ such that 
\begin{equation*}\left\{ 
	\begin{aligned}
	L^{(2)}  v  &= 0  &\text{ in }& \Omega\\
 	v &= u_1  &\text{ on }& \Gamma_1\\
	C^{(2)} v &= 0  &\text{ on }& \Gamma_0.
	\end{aligned} \right.
\end{equation*}
As $u_1 - v = 0$ on $ \Gamma_1$ and $(u_1 - v)_\tau =0$ on $\Gamma_0$  we obtain from the variational formulation that 
\begin{align*}
    0 &= \int_{\Omega} \scal{\sigma_2 (u_1 - v)  }{\varepsilon (u_2)}_{\CC^3 \times \CC^3} \; \dd x .
\end{align*}
Now $\Lambda^{(1)} = \Lambda^{(2)}$ implies 
\begin{align*}
    \int_{\Omega} \scal{\sigma_2 (v)  }{\varepsilon (u_2)}_{\CC^3 \times \CC^3} \; \dd x 
    &=\scal{B^{(2)} v}{u_2|_{\Gamma}} \\
    &=\scal{B^{(1)} u_1}{u_2|_{\Gamma}} \\
    &= \int_{\Omega} \scal{\sigma_1 (u_1)  }{\varepsilon (u_2)}_{\CC^3 \times \CC^3} \; \dd x ,
\end{align*}
which proves the assertion.
\end{proof}
\section{Construction of Complex Geometric Optics Solutions}\label{sec:construction}
In what follows we will construct suitable solutions   $u_i$ of $L^{(i)}  u_i = 0$ in $\Omega$ and $C^{(i)} u_i = 0$ on $\Gamma_0$. The method   is well-known for $\Gamma_0 = \varnothing$ however in the case of partial data, the main difficulty relies on the additional boundary condition. To this end we use the ideas in \cite{Isakov} and use a  reflection argument along  the axis $x_3=0$. 

For $x = (x_1,x_2,x_3)$ we put $x^\vee \coloneqq (x_1,x_2,-x_3)$.
By the assumption on the Lam\'e coefficients, we may extend the functions $\mu_i$ and $\lambda_i$ evenly to $\Omega^{\times 2} \coloneqq \Omega \cup \Gamma_0 \cup \Omega^\vee$,
where 
\[\Omega^\vee  \coloneqq \{ x \in \RR^3 : x^\vee \in \Omega \}.\]
Moreover, we define for $w \in L^2(\Omega^{\times 2};\CC^3)$ the function
\begin{align*}
	w^\vee(x) \coloneqq w(x^\vee)^\vee . 
\end{align*}
\begin{proposition}\label{prop:h-double}
    Assume that $\Lambda^{(1)} = \Lambda^{(2)}$.
    If $w_i \in H^1(\Omega^{\times2};\CC^3)$ are solutions of $L^{(i)} w_i = 0$, then
    \begin{align*}
        H_{\Omega^{\times 2}}(w_1, w_2) =  H_{\Omega^{\times2}}(w_{1}, w_{2}^\vee).
    \end{align*}
\end{proposition}
For the proof we use the following lemma. For the sake of clarity we simply write $L^{(i)} = L$, $C^{(i)} = C$ and $\mu_i = \mu$, $\lambda_i = \lambda$. 
\begin{lemma}\label{lem:extension}
	The following assertions hold true:
	\begin{enumerate}[(i)]
        \item \label{ext-1} If $w \in H^1(\Omega^{\times2};\CC^3)$ satisfies $L w = 0$ in $\Omega^{\times 2}$, then
        $v = w^\vee$ satisfies $L v = 0$ in $\Omega^{\times 2}$.
 		\item If $w \in H^1(\Omega^{\times2};\CC^3)$ satisfies $L w = 0$ in $\Omega^{\times 2}$ and the symmetry condition 
 		\[w = - w^\vee ,\]
 		then the restriction $u = w|_{\Omega}$ satisfies
 		\begin{align*}
			L u =0 \text{ in } \Omega, \qquad C u = 0 \text{ on } \Gamma_0. 
 		\end{align*}
	\end{enumerate}
\end{lemma}
\begin{proof}
    Let $J = \diag(1,1,-1)$. 
    By a straightforward calculation, we see that if $v = w^\vee$, then
    \begin{equation}\label{eq:reflect-ops}
        \begin{aligned}
            \ddiv(v)(x) &= \ddiv(w)(x^\vee), \\
            \ep(v)(x) &= J\ep(w)(x^\vee)J , \\
            \sigma(v)(x) &= J \sigma(w) (x^\vee) J , 
        \end{aligned}
    \end{equation}
    where we have used that $\lambda$ and $\mu$ are extended evenly. Finally, we obtain $Lv(x) = J (Lw)(x^\vee) = 0$, which proves \eqref{ext-1}.

    The second assertion follows easily for $u \in H^2(\Omega;\CC^3)$, because
    \begin{align*}
         Cu = \begin{pmatrix}  u_1 \\ u_2  \\ \lambda (\partial_1 u_1  + \partial_2 u_2) + (\lambda + 2\mu) \partial_3 u_3   \end{pmatrix} \qquad \text{ on } \Gamma_0 .
     \end{align*}
    In the general case $u \in H^1(\Omega;\CC^3)$ we have to apply again the variational formulation. Let $v \in H^{1}(\Omega;\CC^3)$ with $v = 0$ on $ \Gamma_1$ and $v_\tau =0$ on $\Gamma_0$. We have to show that 
    \begin{align*}
        0 &= \int_\Omega \scal{\sigma^{(i)} (u)}{\varepsilon (v) }_{\CC^3 \times \CC^3} \; \dd x . 
    \end{align*}
    To this end we note that we may approximate $v = (v_1, v_2 ,v_3)^T$ in $H^1(\Omega;\CC^3)$ by functions $\phi_k = ( \phi_{1,k} , \phi_{2,k} , \phi_{3,k})^T$ such that $\supp(\phi_{1,k}) , \supp(\phi_{2,k}) \subseteq \Omega$ and $\supp (\phi_{3,k}) \subseteq \Omega \cup \Gamma_0$. Thus, 
    the functions 
    \begin{align*}
        \tilde \phi_k(x)  := \begin{cases} \phi_k(x) , & \quad x \in \Omega , \\ - \phi_k(x^\vee)^\vee , & \quad x \in \Omega^\vee  , \end{cases} 
    \end{align*}
    belong to $C_c^\infty(\Omega;\CC^3)$ and we have 
    \begin{align*}
         \int_\Omega \scal{\sigma^{(i)} (u)}{\varepsilon (v) }_{\CC^3 \times \CC^3} \; \dd x  &= \lim_{k \to \infty} \int_\Omega \scal{\sigma^{(i)} (u)}{\varepsilon (\phi_k) }_{\CC^3 \times \CC^3} \; \dd x \\
         &= \lim_{k \to \infty} \frac12 \int_{\Omega^{\times 2}} \scal{\sigma^{(i)} (w)}{\varepsilon (\tilde \phi_k) }_{\CC^3 \times \CC^3} \; \dd x\\
         &= \lim_{k \to \infty} \frac12 \scal{L^{(i)} v}{\tilde \phi_k} =  0 .
     \end{align*}
    This implies the assertion. 
\end{proof}
\begin{proof}[Proof of Proposition~\ref{prop:h-double}]
Set $u_i \coloneqq w_i - w_i^\vee$ on $\Omega$. Then the  functions $u_i$ solve the boundary value problem \eqref{eq:bvp_Poisson_neu} and by Lemma~\ref{lem:h-omega} we have that $H_\Omega(u_1,u_2) = 0$. Writing this in terms of $w_1, w_2$ yields
\begin{align*}
	0 = H_\Omega(w_{1}|_{\Omega} , w_{2}|_{\Omega} ) + H_\Omega(w^\vee_{1}|_{\Omega} , w^\vee_{2}|_{\Omega} ) - 
	H_\Omega(w_{1}|_{\Omega} , w_{2}^\vee|_{\Omega}) - H_\Omega(w_{1}^\vee|_{\Omega} , w_{2}|_{\Omega}).
\end{align*}
The assertion follows from    \eqref{eq:reflect-ops} as 
\begin{align*}
	H_\Omega(w_{1}|_{\Omega} , w_{2}|_{\Omega} ) + H_\Omega(w^\vee_{1}|_{\Omega} , w^\vee_{2}|_{\Omega} ) &=  H_{\Omega^{\times2}}(w_1, w_2 ) \\
	H_\Omega(w_{1}^\vee|_{\Omega} , w_{2}|_{\Omega}) + H_\Omega(w_{1}|_{\Omega} , w_{2}|_{\Omega}^\vee) &=  H_{\Omega^{\times2}}(w_1, w_2^\vee ) . 
\end{align*}
\end{proof}
Note that in general,  $H_{\Omega^{\times 2}}(w_1,w_2)$ will not vanish for solutions $w_1,w_2$ of $L^{(i)} w_i = 0$ on $\Omega^{\times2}$. However we will construct  a family of  complex geometric optics solutions such that $H_{\Omega^{\times 2}}(w_1,w_2)$ vanishes asymptotically to infinite order, cf. Proposition~\ref{prop:h-double}.
We briefly want to recall the construction of these functions following~\cite{EsRa}. We note that the original method proposed by \cite{SylvesterUhl} for the Laplacian does not directly apply as the principal symbol of the elasticity operator 
\begin{align*}
    \sigma^{2}(L) &= -(\lambda + \mu)\xi \xi^T -\mu \|\xi\|^2 \id_3.
\end{align*}
is not of diagonal type. Moreover, due to the matrix structure of the operator the construction is more involved. 

The following reduction is due to Ikehata (cf. \cites{UhlmannDev,UhlWangWu}).
Consider the equation
\begin{align}\label{eq:m-plain}
    M \begin{pmatrix} h \\ f\end{pmatrix} \coloneqq \Delta \begin{pmatrix} h \\ f \end{pmatrix} + V_1(x) \begin{pmatrix} \nabla f \\ \nabla \cdot h\end{pmatrix} + V_0(x) \begin{pmatrix} h \\ f\end{pmatrix} = 0,
\end{align}
where $V_1$ and $V_0$ are given by
\begin{align*}
    V_1(x) &= \begin{pmatrix} 2 \mu^{1/2} (-\nabla^2 + \id_3\Delta) \mu^{-1} & - \mu^{-1} \nabla \mu \\ 0 & \frac{\lambda + \mu}{\lambda + 2 \mu} \mu^{1/2} \end{pmatrix},
    &\intertext{and}
    V_0(x) &= \begin{pmatrix} -\mu^{-1/2} (-2\nabla^2 + \id_3\Delta) \mu^{1/2} & -2\mu^{-5/2} (-\nabla^2 + \id_3\Delta) \mu \nabla \mu \\ -\frac{\lambda - \mu}{\lambda + 2\mu} (\nabla \mu)^T & -\mu \Delta \mu^{-1} \end{pmatrix}.
\end{align*}
The function
\begin{align}\label{ansatz}
	 w \coloneqq \mu^{-1/2} h   + \mu^{-1} \nabla f  - f  \nabla \mu^{-1} ,
\end{align}
solves $L w =0$ if $(h,f)^T$ is a solution of \eqref{eq:m-plain}.
Note that the principal symbol of $M$ equals  $-\|\xi\|^2 \cdot \id_4$. 
\begin{remark}\label{rem:missing-delta}
    Note that the term $\id_3\Delta \mu^{-1}$ in $V_1$ is missing in \cite{EsRa}, but we will see that this term is irrelevant for our considerations.
\end{remark}
Let 
\begin{align*}
    Z &\coloneqq \{ \theta \in \CC^3 \setminus\{0\} \colon \theta \cdot \theta = 0\} .
\end{align*}
In what follows we  construct  solutions of \eqref{eq:m-plain} of the form  $(h,f) = e^{i \zeta  \cdot x} (r, s)$, $\zeta \in Z$.

As a first step we extend $\mu$ and $\lambda$ smoothly to some ball $B_R \coloneqq \{ x \in \RR^3 : |x| < R\}$, which contains $\Omega^{\times 2}$.
The main assumption on the coefficients ensures that this is always possible.
We denote by $\psi\in C^\infty_c(\RR^3)$ a cut-off function with $\supp (\psi) \subset B_R$ and $\psi|_{\Omega^{\times 2}} =1$. Then any solution $(r,s)$ of the differential equation 

\begin{align}\label{eq:m-zeta}
	0  =  M_\zeta \begin{pmatrix} r \\ s \end{pmatrix} \coloneqq \Delta_\zeta  \begin{pmatrix} r \\ s \end{pmatrix} + \psi(x) V_1(x) \begin{pmatrix} i\zeta s + \nabla s \\ i\zeta r + \ddiv(r)    \end{pmatrix} + 
		\psi(x) V_0(x) \begin{pmatrix} r \\ s \end{pmatrix} ,
\end{align}
on $\Omega^{\times 2}$ will give a solution $(h,f)$ of \eqref{eq:m-plain}. Here  $\Delta_\zeta = \Delta + 2i\zeta \cdot \partial_x$.
Choose $\theta \in Z$ and $\ell_0 \in \RR^3$ such that $\ell_0 \cdot \theta = 0$.
We define
\begin{align}
    \zeta(\tau) &\coloneqq \frac12 \ell_0 + \tau \theta + \tau\rho(\tau) \Re \theta,\label{eq:zeta}
\end{align}
where
\begin{align}\label{eq:def-rho}
    \rho(\tau) \coloneqq \left( 1 - \frac{|\ell_0|^2}{4\tau^2}\right)^{1/2} - 1 = - \frac{|\ell_0|^2}{4\tau^2} + \mathcal{O}(\tau^{-4}).
\end{align}
For the sake of simplicity we often write  $\zeta$ instead of $\zeta(\tau)$. The remaining steps are well-known, see e.g. \cite{EsRa}. We define $A : B_R \to (\RR^{4\times 4})^3 $ by 
\begin{align*}
	A(x) \cdot \zeta \begin{pmatrix} r \\s \end{pmatrix} &\coloneqq \psi(x) V_1(x) \begin{pmatrix} \zeta s   \\ \zeta \cdot r \end{pmatrix} 
\end{align*}
and let 
\begin{align*}
	\ell(\tau) &\coloneqq 2(\zeta(\tau) - \tau \theta) = \ell_0 + 2 \tau \rho(\tau) \Re \theta.
\end{align*}
Note that 
\begin{align*}
	M_\zeta = M + i(2 \zeta \cdot \partial_x + A \cdot \zeta) =  M_{\frac12\ell(\tau)} + i \tau (2 \theta \cdot \partial_x + A \cdot \theta) .
\end{align*}
To find a solution of $M_\zeta v =0$, we make the formal ansatz 
\[ v   = v_0 + v_1 + \ldots + v_n + \ldots  , \]
and seek to construct $v_n$ such that $\|v_n\|_{H^k(B_R;\CC^4)} = \mathcal O(\tau^{-n})$ as $\tau \to \infty$. To achieve this we look for solutions to  
\begin{align}\label{eq:construction_v_0}
	i \tau (2 \theta \cdot \partial_x + A \cdot \theta) v_0 &= 0 \\
	i \tau (2 \theta \cdot \partial_x + A \cdot \theta) v_n &= - \psi M_{\frac12 \ell(\tau)} v_{n-1} , \qquad  n \ge 1 . \label{eq:construction_v_n}
\end{align}
From \cite{Eskin}*{Theorem 2.1} we have:
\begin{lemma}\label{lem:dbar}
	There exists a matrix $C_0 = C_0(x, \theta)$ depending smoothly on $x$ and $\theta$, which satisfies 
    \begin{align}\label{eq:dbar}
		2(\theta \cdot \partial_x) C_0 + A \cdot \theta C_0 = 0 . 
 	\end{align}
	Moreover, we may choose $C_0$ such that it is invertible for all $(x,\theta) \in B_r \times Z$ and $C_0$ is homogeneous of degree $0$ in $\theta$.
\end{lemma}
In particular, if $g(z)$ is any vector of polynomials in the complex variable $z$, then we may choose
\begin{align*}
	v_0 &= C_0(x, \theta) g(x \cdot  \theta) . 
\end{align*}

Note that the operator $\theta \cdot \partial_x$ can be transformed into the $\bar\partial$-operator by a linear change of coordinates.
In fact, for $\theta = (1,i,0)^T$, we have that $\theta \cdot \partial_x = \partial_{x_1} + i\partial_{x_2}$.
Denote by $\hat{f}$ the Fourier transform of $f$ and
let \[(\Pi_\theta f)(x) = (2\pi)^{-3} \int_{\RR^3} \frac{e^{ix\eta}\hat{f}(\eta)}{i\eta \cdot \theta} d \theta\]
be the Fourier multiplier with symbol $\eta \mapsto (i\eta\cdot\theta)^{-1}$.
Then $\Pi_\theta$  is the inverse of the differential operator $\theta \cdot \partial_x$, i.e.
\begin{align*}
    \Pi_\theta (\theta \cdot \pa_x f)(x) &= f(x),\\
    \theta \cdot \pa_x (\Pi_\theta g)(x) &= g(x)
\end{align*}
for $f,g \in C^\infty_c(\RR^3)$.
Thus, we may put
\begin{align*}
	v_n \coloneqq -(i\tau)^{-1} C_0(x,\theta) \Pi_\theta  \left( C_0^{-1} (x, \theta) \psi(x)  M'_{\frac12\ell(\tau)} v_{n-1} \right)
\end{align*}
and $v_n$ will satisfy \eqref{eq:construction_v_n}. In particular we obtain 
\begin{align}
	\notag M_\zeta (v_0 + \ldots + v_n)  &= \Bigr\{  M_{\frac12\ell(\tau)} + i \tau (2 \theta \cdot \partial_x + A \cdot \theta) \Bigr\}  ( v_0 + \ldots v_n) + M_\zeta' v_n \\
	&= (1-\psi) \sum_{i=0}^{n-1}  M_{\frac12\ell(\tau)} v_i + M_{\frac12 \ell(\tau)} v_n . 
	\label{eq:M_zeta(v)}
\end{align}
Note that $1- \psi$ vanishes on $\Omega^{\times 2}$. Moreover, comparing the orders in $\tau$ and using that 
\begin{align}\label{eq:pi-sobolev}
    \Pi_\theta: H^s_{\mathrm{comp}} (\RR^3) \rightarrow H^s_{\mathrm{loc}}(\RR^3) ,
\end{align}
we obtain  for any $k \in \NN$ that 
\begin{align}\label{eq:vn-estimate}
	\| v_n \|_{H^k(B_R;\CC^4)} = \mathcal O(\tau^{-n})  \qquad \text{as } \tau \to \infty.  
\end{align}
Then we obtain as in \cite{EsRa} (cf. also \cite{Eskin}*{Sect. 3} and \cite{NakUhl02}):
\begin{lemma}\label{lem:remainder}
	For large $\tau>0$ there exists $v_e \in H^k(B_R;\CC^4)$ which satisfies
	\begin{align*}
		M_\zeta v_e = - \psi M_{\frac12 \ell(\tau)} v_n \qquad \text{in } B_R
	\end{align*}
	and  $\| v_e \|_{H^k(B_r;\CC^4)} = \mathcal O(\tau^{-n-1})$ as $\tau \to \infty$. 
\end{lemma}
Let $v_e$ be chosen as above. Setting
\begin{align*}
	v \coloneqq v_0 + v_1 + \ldots + v_n + v_e 
\end{align*}
we obtain  that 
\begin{align*}
	M_\zeta v  &= (1- \psi) \sum_{i=0}^n  M_{\frac12\ell(\tau)} v_i  , 
\end{align*}
and in particular  $M_\zeta v = 0$ in  $\Omega^{\times2}$. 
As a next step we use these complex geometric optics solutions to extract information from Proposition~\ref{prop:h-double}.
To this end we choose
\begin{align*}
    \zeta(\tau) &\coloneqq \frac12 \ell_0 + \tau \theta + \tau\rho(\tau) \Re \theta,
\intertext{and}
    \zeta'(\tau) &\coloneqq -\frac12 \ell_0 + \tau \bar\theta + \tau\rho(\tau) \Re \theta.
\end{align*}
Then we obtain solutions $v^{(i)} \coloneqq (r^{(i)} , s^{(i)})^T \in C^\infty(\Omega^{\times2}; \CC^4)$ such that 
\begin{align*}
	M_{\zeta(\tau)} ^{(1)} v^{(1)} = 0 \quad \text{and} \quad  M_{\zeta'(\tau)}^{(2)} v^{(2)} = 0 \quad \text{in } \Omega^{\times 2} . 
\end{align*}
Here $M_\gamma^{(i)}$, $\gamma \in \CC^3$, $i=1,2$, is the corresponding operator corresponding to the pair $(\lambda^{(i)}, \mu^{(i)})$. Let 
\begin{align*}
 	w_1  &\coloneqq  \mu_1^{-1/2} e^{i\zeta \cdot x}  r^{(1)}     + \mu_1^{-1} \nabla \bigl(e^{i\zeta \cdot x} s^{(1)} \bigr)
		- \bigl(e^{i\zeta \cdot x} s^{(1)} \bigr)  \nabla \mu_1^{-1} ,\\
	w_2 &\coloneqq \mu_2^{-1/2} e^{i\zeta' \cdot x}r^{(2)}    + \mu_2^{-1} \nabla \bigl(e^{i\zeta' \cdot x} s^{(2)} \Bigr)
		- \bigl(e^{i\zeta' \cdot x} s^{(2)} \bigr)  \nabla \mu_2^{-1} . 
\end{align*}
Then $L^{(i)} w_i = 0$ in $\Omega^{\times 2}$ and
by Proposition~\ref{prop:h-double}, we have
\begin{align*}
    H_{\Omega^{\times 2}}(w_1, w_2) = H_{\Omega^{\times 2}}(w_{1} , w_{2}^\vee) . 
\end{align*}
If we set
\begin{align*}
    Z_0 = \{\theta \in Z \colon \Im \theta_3 = 0\},
\end{align*}
then we obtain the following result:
\begin{lemma}\label{lem:asymptotic-terms}
    Let $\theta \in Z_0$ and $\ell_0 \in \RR^3$ such that $\theta \cdot \ell_0 = 0$ and $\Re\theta_3 \not = 0$.
    Then
    \begin{align*}
        H_{\Omega^{\times 2}}(w_1, w_2) = \mathcal O(\tau^{-\infty}) , \quad \text{as } \tau \to \infty. 
    \end{align*}
\end{lemma}

\begin{proof}
By the previous considerations, it suffices to show that
\begin{align*}
    H_{\Omega^{\times 2}}(w_{1} , w_{2}^\vee) = \mathcal O(\tau^{-\infty}).
\end{align*}
Consider the functions $p_i , q_i$ defined by 
\begin{align*}
    p_1 &= e^{-i\zeta \cdot x} \ddiv (w_1) , & q_1 &= e^{-i\zeta \cdot x} \varepsilon (w_1) , \\
	p_2 &= e^{-i\zeta'\cdot x} \ddiv(w_2) , & q_2 &= e^{i\zeta' \cdot x} \varepsilon (w_2) .
\end{align*}
From \eqref{eq:vn-estimate} and Lemma~\ref{lem:remainder}, we obtain for any $k \in \NN$ that
\begin{equation}\label{eq:pq-bound}
    \|p_1 \overline{p_2} \|_{H^k(\Omega^{\times2};\CC)} = \mathcal O(\tau^4) , \qquad \|\scal{q_1}{q_2}_{\CC^{3\times3}} \|_{H^k(\Omega^{\times2};\CC)} = \mathcal O(\tau^4) .
\end{equation}
Putting
\begin{align*}
	  F (x) &\coloneqq \{ \lambda_2 - \lambda_1 \}  p_{1} (x) \overline{p_{2}(x^\vee)}  + 2\{\mu_2 - \mu_1\}  \scal{q_1(x)}{Jq_2(x^\vee) J}_{\CC^3 \times \CC^3},
\end{align*}
we have 
\begin{align*}
	H_{\Omega^{\times 2}} (w_{1} , w_{2}^\vee) &=\int_{\Omega^{\times 2}}  e^{i (\zeta - \overline{\zeta'}^\vee) \cdot x} F(x) \; \dd x .
\end{align*}
We note that the assumption $\theta \in Z_0$ implies that $(\zeta - \overline{\zeta'}^\vee)_3 = 2\tau (1+ \rho(\tau)) \Re \theta_3$.
Hence, by partial integration for  $\Omega$ and $\Omega^\vee$ we obtain  
\begin{align*}
	H_{\Omega^{\times 2}} (w_{1} , w_{2}^\vee)   &=  - \frac{1}{ 2i \tau (1 + \rho(\tau))\Re\theta_3}   \int_{\Omega^{\times 2}} e^{i (\zeta - \overline{\zeta'}^\vee) \cdot x} \partial_3 F(x) \; \dd x\\
	&\quad + \int_{\partial \Omega} e^{i (\zeta - \overline{\zeta'}^\vee) \cdot x} \normal_3\partial_3 F(x) \; \dd x  + \int_{\partial (\Omega^\vee)} e^{i (\zeta - \overline{\zeta'}^\vee) \cdot x} \normal_3\partial_3 F(x) \; \dd x. 
\end{align*}
Using \cite{NakUhl95} we have that $\mu_1 = \mu_2$ and $\lambda_1 = \lambda_2$ on $\Gamma_1$ to infinite order.
Then the symmetry assumptions imply that the boundary integrals vanish. By induction we have 
\begin{align*}
	H_{\Omega^{\times 2}} (w_{1} , w_{2}^\vee)   &= \left( - \frac{1}{ 2i \tau(1 + \rho(\tau)) \Re\theta_3} \right)^n  \int_{\RR^3} e^{i (\zeta - \overline{\zeta'}^\vee) \cdot x} \partial_3^n F(x) \; \dd x =  \mathcal O(\tau^{-n}), 
\end{align*}
since $(1 + \rho(\tau))^{-1} = 1 + O(\tau^{-1})$.
\end{proof}

\section{Proof of Theorem~\ref{thm:main1} and Theorem~\ref{thm:main2}}\label{sec:proof-bdd}
\subsection{The Leading Order Term of the Form \texorpdfstring{$H$}{H}}
As a next step we want to apply the results by Eskin and Ralston. To this end we set 
\begin{align*}
	a_i \coloneqq (\theta \cdot \partial_x)^2 \mu_i^{-1} , \qquad b_i \coloneqq \frac{\mu_i}2 \cdot \frac{\lambda_i + \mu_i}{\lambda_i + 2\mu_i} 
\end{align*}
as well as 
\begin{align*} 
	 \begin{pmatrix} R_0^{(i)} \\ s_0^{(i)} \end{pmatrix} (x,\theta) &\coloneqq  \begin{pmatrix} \mu_i^{-1} \theta^T & 0 \\ 0 & 1 \end{pmatrix} C^{(i)}_0(x,\theta) g_i(x \cdot \theta) , \qquad i =1 ,2 .
\end{align*}
Here $C^{(i)}_0(x,\theta)$ is chosen as in Lemma \ref{lem:dbar} and $g_i(z)$ is a vector of polynomials in $z$. 
By the same calculation as in \cite{EsRa}, we obtain:\footnote{The term mentioned in Remark~\ref{rem:missing-delta} vanishes because $\theta \cdot \theta = 0$.}
\begin{lemma}[cf. \cite{EsRa}*{Lemma 2.1}]\label{lem:dbar-h2}
We have $H_{\Omega^{\times 2}}(w_1, w_2)  = \tau^2 H_2 + \mathcal O(\tau)$, where  
\begin{align*}
	H_2 = \int_{\Omega^{\times 2}} e^{i \ell_0 \cdot x} ( \overline{R_0}^{(2)}, \overline{s_0}^{(2)} )(x,\theta)  V(x,\theta) \begin{pmatrix} R_0^{(1)} \\ s_0^{(1)} \end{pmatrix} (x,\theta) \; \dd x  , 
\end{align*}
and 
\begin{align*}
    V(x,\theta) \coloneqq \begin{pmatrix} ( \lambda_1 + \mu_1 - (\lambda_2 + \mu_2)) \frac{\mu_1 \mu_2}{(\lambda_1 + 2\mu_1)(\lambda_2 + 2\mu_2)} 
	& 2(\mu_2^{-1} -\mu_1^{-1}) (\theta \cdot \partial_x) b_2 \\[5pt] 2(\mu_2^{-1} -\mu_1^{-1}) (\theta \cdot \partial_x) b_1 
	& 2 (\mu_2^{-1} -\mu_1^{-1}) (b_1 a_1 + b_2 a_2) \end{pmatrix} . 
\end{align*}
\end{lemma}
Using Lemma \ref{lem:asymptotic-terms} we immediately obtain:
\begin{coro}
    Let $\theta \in Z_0$ and $\ell_0 \in \RR^3$ such that $\theta \cdot \ell_0 = 0$.
    If $\Re(\theta_3) \neq 0$, then
	\begin{align}\label{eq:H_2_C}
        0&= \int_{\Omega^{\times 2}} e^{i \ell_0 \cdot x} ( \overline{R_0}^{(2)}, \overline{s_0}^{(2)} )(x,\theta)  V(x,\theta) \begin{pmatrix} R_0^{(1)} \\ s_0^{(1)} \end{pmatrix} (x,\theta) \; \dd x . 
    \end{align}
\end{coro}
By continuity, the assertion also holds true for $\Re(\theta_3) = 0$. The remaining part of the proof may be deduced as in \cite{EsRa}, taking into account that we have to assume that $\Im(\theta_3) = 0$. For the sake of completeness we want to recall the main steps. We note that  $(R_0^{(i)} , s_0^{(i)})^T$ may be written as 
\begin{align}
    \begin{pmatrix} R_0^{(i)} \\ s_0^{(i)} \end{pmatrix}(x,\theta)  = \tilde{C}^{(i)}(x,\theta) \tilde g_i(x\cdot \theta) , 
\end{align}
where $\tilde g_i(z)$ is a 2-vector of polynomials and $\tilde{C}^{(i)} = \tilde{C}^{(i)}(x,\theta)$ satisfies the equation 
the equation
\begin{align}\label{eq:dbar-two}
    (\theta \cdot \partial_x) \tilde{C}^{(i)} = \tilde{A}^{(i)} \tilde{C}^{(i)} . 
\end{align}
Here $\tilde{A}^{(i)} = \tilde{A}^{(i)}(x,\theta)$ is given by
\begin{align*}
    \tilde{A}^{(i)} = \begin{pmatrix} 0 & - a_i \\ b_i & 0 \end{pmatrix} . 
\end{align*}
Again, we apply \cite{Eskin}*{Theorem 2.1} to obtain a solution $\tilde{C}^{(i)}(x,\theta)$ of \eqref{eq:dbar-two} defined for $|\Re \theta| = |\Im \theta| = 1$
and that is invertible for $x < |R|$.

We observe that we can smoothly extend $\tilde{C}^{(i)}$ to $\theta \in Z$ by requiring the homogeneity
\begin{align*}
    \tilde{C}^{(i)}(x,r \theta) = \begin{pmatrix}1 & 0\\0 & r^{-1}\end{pmatrix}\tilde{C}^{(i)}(x,\theta), \quad r>0.
\end{align*}
Then we may consider in Equation \eqref{eq:H_2_C} solutions of \eqref{eq:dbar-two} instead of \eqref{eq:dbar} and  we obtain:
\begin{proposition}[\cite{EsRa}*{Sect. 3}]\label{prop:upgrade-soln}
    Let $\tilde C^{(i)}(x, \theta)$ be a solution of \eqref{eq:dbar-two}, which is invertible for $|x| < R$ and let $\tilde{g}_{i}$  be  any 2-vectors of polynomials. Then we have 
    \begin{align*}
        0 &=  \int_{\Omega^{\times 2}} e^{i\ell_0 \cdot x} \overline{\tilde{g}_2(\theta \cdot x)}^t (\Ctwo)^*(x,\bar\theta) V(x,\theta) \Cone(x,\theta) \tilde{g}(x\cdot \theta) dx , 
    \end{align*}
    where  $\theta \in Z_0$, $\ell_0 \in \RR^3$ are chosen such that $\theta \cdot \ell_0 = 0$.
\end{proposition}
\subsection{Parametrizing the Complex Geometric Optics Solutions}\label{sec:parametrize}
Since \[\left( \frac{\Re\theta}{|\Re\theta|}, \frac{\Im\theta}{|\Im\theta|}, \frac{\ell_0}{|\ell_0|}\right)\] forms an orthonormal basis, we change the coordinates into this basis.
To this end, write $y = (y_1,y_2,y_3)$ with $y_1 = (\Re\theta/|\Re\theta|) \cdot x$, $y_2 = (\Im\theta/|\Im\theta|) \cdot x$, and $y_3 = (\ell_0/|\ell_0|) \cdot x$.

We extend $V(\cdot,\theta)$ by zero to $\RR^3$ and by \cite{NakUhl95} and the assumptions on the Lam\'e parameters $\lambda_j$ and $\mu_j$, this extention is smooth.
From Proposition~\ref{prop:upgrade-soln} we obtain by applying the Fourier transform that for any 2-vectors of polynomials $g_1,g_2$ and $\theta \in Z_0$,
\begin{align}\label{eq:transformed-h}
    0 = \int_{\RR^2} \overline{g_2(\theta\cdot y)}^t (\tilde{C}_2)^*(y,\overline{\theta}) V(y,\theta) \tilde{C}_1(y,\theta) g_1(\theta \cdot y) \;\dd y_1 \;\dd y_2, \qquad y_3 \in \RR.
\end{align}

Recall that $\Pi_\theta$ is the Fourier multiplier with $\eta \mapsto (i\eta \cdot \theta)^{-1}$.
Set
\begin{align}
    B_0(y,\theta) &= \Ctwo(y,\overline{\theta})^* V(y,\theta) \Cone(y,\theta), \label{eq:b-zero}\\
    B(y,\theta) &= \Pi_\theta B_0(y,\theta), \label{eq:b}\\
    B'(y,\theta) &= (\Ctwo(y,\overline{\theta})^*)^{-1} B(y,\theta) \Cone(y,\theta)^{-1}. \label{eq:b-prime}
\end{align}

The reason for defining $B'$ is that it has a particular simple form:
\begin{proposition}\label{prop:b-prime}
    There exist functions $b_{12},b_{21},b_{22}$ on $B_R$ such that for all $\theta \in Z_0$,
    \begin{align*}
        B'(x,\theta) = \begin{pmatrix}0 & b_{12}(x)\\b_{21}(x) & b_{22}(x) \cdot \theta\end{pmatrix}.
    \end{align*}
\end{proposition}
Before we prove the proposition, we state the main result of this section:
\begin{proposition}\label{prop:main}
    For all $\theta \in Z$ and $x \in \Omega$, we have the following set of equations
    \begin{equation*}
    \left\{\begin{aligned}
        -b_2 b_{21} - b_1 b_{12} &= v_{11}\\
        \theta \cdot \partial_x b_{12} - b_2 b_{22}\cdot \theta &= v_{12}\\
        \theta \cdot \partial_x b_{21} - b_1 b_{22}\cdot \theta &= v_{21}\\
        \theta \cdot \partial_x b_{22} \theta + a_2 b_{12} + a_1 b_{21} &= v_{22},
    \end{aligned}\right.
    \end{equation*}
    where the $v_{jk}$ are the entries of the matrix $V = V(x,\theta)$ as in Lemma~\ref{lem:dbar-h2}.
\end{proposition}

To prove Proposition~\ref{prop:b-prime}, we have to parametrize the space $Z_0$ to extract information from equation \eqref{eq:transformed-h}.
Let \[\xi(t) = \left((1/2) ( t - t^{-1}), (i/2)(t+t^{-1}), 1\right)^T.\]
Then $Z_0 = \{ r\xi(t) \colon r \in \RR_+, t \in \CC\setminus\{0\}\}$.

From the homogeneity of $V$ and $C$, we can deduce the homogeneity of the matrix entries of $B'$.
\begin{lemma}\label{lem:homogeneity}
    The coefficients $B_{ij}'(x,\theta)$ are positive homogeneous of degree $i + j -3$ in $\theta$,
    \begin{align*}
        B_{ij}'(x,r\theta) = r^{i+j-3}B_{ij}'(x,\theta), \quad r > 0.
    \end{align*}
\end{lemma}
\begin{lemma}\label{lem:holomorphy}
    $\bar\partial_t B'(y,r\xi(t)) = 0$ for $t \in \CC\setminus\{0\}$.
\end{lemma}
\begin{proof}
    By homogeneity, we may assume that $r = 1$.
    First, we will calculating $\bar\partial_t B_0(y,\xi(t))$.
    For this, we set
    \begin{align*}
        M_1(y,t) &= \Cone(y,\xi(t))^{-1} \bar\partial_t \Cone(y,\xi(t)),\\
        M_2(y,t) &= \bar\partial_t (\Ctwo(y,\overline{\xi(t)})^*) \,(\Ctwo(y,\overline{\xi(t)})^*)^{-1},
    \end{align*}
    then obtain
    \begin{align*}
        \bar\partial_t B_0(y,\xi(t)) &= \bar\partial_t (\Ctwo(y,\overline{\xi(t)})^*) \, V(y,\xi(t)) \Cone(y,\xi(t))\\
        &\phantom{=} + \Ctwo(y,\overline{\xi(t)})^* \, \bar\partial_t V(y,\xi(t)) \Cone(y,\xi(t))\\
        &\phantom{=} + \Ctwo(y,\overline{\xi(t)})^* \, V(y,\xi(t)) \bar\partial_t \Cone(y,\xi(t))\\
        &= M_2(y,t) B_0(y,\xi(t)) + B_0(y,\xi(t)) M_1(y,t).
    \end{align*}
    The second summand vanishes because $V(y,\xi(t))$ is holomorphic in $t \in \CC\setminus\{0\}$.

    We now claim that
    \begin{align}\label{eq:tbar-b}
        \bar\partial_t B(y,\xi(t)) = M_2(y,t) B(x,\xi(t)) + B(y,\xi(t)) M_1(y,t).
    \end{align}
    If we choose $g_1$ and $g_2$ the basis vectors, then it follows from \eqref{eq:transformed-h} that
    \begin{align*}
        \int_{\RR^2} B_0(y,\theta) dy_1 dy_2 = 0, \qquad y_1 \in \RR
    \end{align*}
    and together with Lemma 6.1 from Eskin~\cite{Eskin} this implies that $(\bar\partial_t \Pi_{\xi(t)}) B_0(y,\xi(t))$ vanishes for $t \in \CC\setminus\{0\}$ and hence
    \begin{align*}
        \bar\partial_t B(y,\xi(t))) &= (\bar\partial_t \Pi_{\xi(t)}) B_0(y,\xi(t)) + \Pi_{\xi(t)} (\bar\partial_t B_0(y,\xi(t))\\
        &= \Pi_{\xi(t)} \big( M_2(y,t) B_0(y,\xi(t)) + B_0(y,\xi(t)) M_1(y,t) \big).
    \end{align*}
    From the definition of $\Cone$ and $\Ctwo$, we see that
    \begin{align*}
        (\xi(t) \cdot \pa_y) M_j(y,t) = 0, \text{ for } j=1,2.
    \end{align*}
    This implies (cf. \cite{Eskin}*{p.~527}) that
    \begin{align*}
        M_j(y,t)\Pi_{\xi(t)}B_0(y,\xi(t)) - \Pi_{\xi(t)} \big(M_j(y,t) B_0(x,\xi(t))\big) = 0.
    \end{align*}
    Hence, we obtain the claimed equality~\eqref{eq:tbar-b}.

    Finally, we calculate
    \begin{align*}
        \bar\partial_t B'(y,\xi(t)) &= \bar\partial_t \left( (\Ctwo(y,\overline{\xi(t)})^*)^{-1} B(y,\xi(t)) \Cone(y,\xi(t))^{-1}\right)\\
        &= \bar\partial_t (\Ctwo(y,\overline{\xi(t)})^*)^{-1} \, B(y,\xi(t)) \, \Cone(y,\xi(t))^{-1}\\
        &\phantom{=} + (\Ctwo(y,\overline{\xi(t)})^*)^{-1} \, \bar\partial_t B(y,\xi(t)) \, \Cone(y,\xi(t))^{-1}\\
        &\phantom{=} + (\Ctwo(y,\overline{\xi(t)})^*)^{-1} \, B(y,\xi(t)) \, \bar\partial_t \Cone(y,\xi(t))^{-1}\\
        &= -(\Ctwo(y,\overline{\xi(t)})^*)^{-1} M_2(x,t)  B(y,\xi(t)) \, \Cone(y,\xi(t))^{-1}\\
        &\phantom{=} + (\Ctwo(y,\overline{\xi(t)})^*)^{-1} \big( M_2(y,t) B(y,\xi(t)) + B(y,\xi(t))M_1(y,t) \big) \Cone(y,\xi(t))^{-1}\\
        &\phantom{=} - (\Ctwo(y,\overline{\xi(t)})^*)^{-1} \, B(y,\xi(t)) M_1(y,t) \, \Cone(y,\xi(t))^{-1}\\
        &= 0.
    \end{align*}
\end{proof}
\begin{proof}[Proof of Proposition~\ref{prop:b-prime}]
    Write
    \begin{align*}
        B'(x,\theta) = \begin{pmatrix} B_{11}'(x,\theta) & B_{12}'(x,\theta) \\ B_{21}'(x,\theta) & B_{22}'(x,\theta)\end{pmatrix}.
    \end{align*}
    For $i,j \in \{1,2\}$, we have by continuity of $B'(x,\theta)$ that
    \begin{align*}
        \sup_{|\theta| = 1} |B_{ij}'(x,\theta)| < \infty
    \end{align*}
    for $\theta \in Z_0$.

    First, we consider the entries $B_{12}'$ and $B_{21}'$. From Lemma \ref{lem:homogeneity} it follows that
    \begin{align*}
        |B_{12}'(x,\xi(t))| &= |B_{12}'(x,\xi(t)/|\xi(t)|)| < C < \infty,\\
        |B_{21}'(x,\xi(t))| &= |B_{21}'(x,\xi(t)/|\xi(t)|)| < C' < \infty.
    \end{align*}
    By Lemma~\ref{lem:holomorphy} these are therefore bounded entire functions in $t$, hence Liouville's theorem
    implies that $B_{12}'$ and $B_{21}'$ are constant in $\xi(t)$.
    
    For $B_{11}'$ we have
    \begin{align*}
        |B_{11}'(x,\xi(t))| = |B_{11}'(x,\xi(t)/|\xi(t)|)| \cdot |\xi(t)|^{-1}.
    \end{align*}
    Therefore, $B_{11}'(x,\xi(t))$ is entire and tends to $0$ for $|t|\to \infty$ and $|t| \to 0$.
    This implies that $B_{11}'(x,\xi(t)) = 0$.

    By the same arguments, we obtain for $B_{22}'(x,\xi(t))$ that
    \begin{align*}
        B_{22}'(x,\xi(t)) = B_{22,-1}'(x) t^{-1} + B_{22,0}'(x)  + B_{22,1}'(x) t.
    \end{align*}
    Setting
    \begin{align*}
        b_{22}'(x) = \begin{pmatrix} B_{22,1}'(x) - B_{22,-1}'(x) \\ -i (B_{22,1}'(x) + B_{22,-1}'(x)) \\ B_{22,0}'(x)\end{pmatrix},
    \end{align*}
    we obtain $B'(x,\theta) = b_{22}(x) \cdot \theta$.

\end{proof}
\begin{proof}[Proof of Proposition~\ref{prop:main}]
    Applying $\theta \cdot \partial_x$ to \eqref{eq:b-prime}, we have that
    \begin{align*}
        \theta \cdot \partial_x B'(x,\theta) + (A^{(2)})^* B'(x,\theta) + B'(x,\theta) A^{(1)} = V(x,\theta).
    \end{align*}
    Using Proposition~\ref{prop:b-prime}, we obtain the claim for $x \in B_R$ and $\theta \in Z_0$.

    In the case that $\Im \theta_3 \neq 0$,
    we observe that the first equality is independent of $\theta$, the second and third are linear in $\theta$ and the last one is quadratic in $\theta$.
    Hence, we can multiply by a complex number to reduce the case $\Im \theta_3 \neq 0$ to $\Im \theta_3 = 0$.
\end{proof}

\section{From Bounded Domains to Infinite Cylinders}\label{sec:cylinder}
Let as before $\Omega \subseteq \RR^2 \times \RR_{>0}$. We recall the corresponding assumptions. There exists  $G\subseteq \RR^2$ such that
\begin{align*}
	\Omega \cap (\RR^3 \setminus [-R,R]^3) = (\RR \setminus [-R,R]) \times G 
\end{align*}
and we suppose that there exists a compact set $K \subseteq \overline{\Omega}$ such that 
\begin{equation}\label{eq:lame-infty}
    \left\{ \begin{aligned} 
        \mu_1(x) &= \mu_2(x) =: \mu(x_2,x_3) \quad \text{for } x = (x_1,x_2,x_3) \in \overline{\Omega} \setminus K \\
        \lambda_1(x) &= \lambda_2(x) =: \lambda(x_2,x_3) \quad \text{for } x = (x_1,x_2,x_3) \in \overline{\Omega} \setminus K.
    \end{aligned} \right.
\end{equation} 
Moreover, we assume that $\Gamma_1 = \partial \Omega \cap (\RR^2 \times \{0\})$ is bounded.
Let  $k \in \RR$ be fixed. We consider the boundary value problem 
\begin{align}\label{eq:bvp_Poisson_cyl} 
	\left\{ \begin{aligned}
	(L^{(i)} -k^2) u  &= f \text{ in } \Omega \\
	u &= g  \text{ on } \Gamma_1 \\
	C^{(i)} u &= 0  \text{ on } \Gamma_0  ,
\end{aligned} \right. 
\end{align}
for suitable $f : \Omega \to \CC^3$ and $g : \Gamma_1 \to \CC^3$. In order to ensure unique solvability, we have to impose the correct radiation conditions and consider exponentially weighted Sobolev spaces. We define for $\beta  \in \RR$ and $s \in \RR$ the spaces 
\begin{align*}
	H^{s}_\beta (\Omega;\CC^3) & \coloneqq \Bigr\{ f  \in H^s_{\mathrm{loc}} (\Omega;\CC^3) : 
		e^{- \beta_1 |x_1| }  \chi (x_1) f(x) \in H^s(\Omega;\CC^3) \Bigr\} , \\
	H^s_\beta (\partial \Omega;\CC^3) & \coloneqq \Bigr\{ g  \in H^s_{\mathrm{loc}} (\partial \Omega;\CC^3) : 
		e^{- \beta_1 |x_1| }  \chi (x_1)  g(x) \in H^s(\partial \Omega;\CC^3) \Bigr\} . 
\end{align*}
Here  $\chi \in C^\infty(\RR)$ shall satisfy $\chi = 0$ in some neighbourhood of $0$ and  $\chi  =1$ outside some compact set. For $s=0$ we write $H^0_{\beta}(\Omega;\CC^3)=L^2_\beta(\Omega;\CC^3)$. Note that 
\[ H^s_\beta (\partial \Omega;\CC^3) = H^{-s}_{-\beta} (\partial \Omega;\CC^3)^* . \]
Without a major effort we may generalise the definitions in Section \ref{sec:dtn} and we obtain the spaces
\[ 
     H^s_\beta(\Gamma_0; T\Gamma_0), H^s_{00,\beta}(\Gamma_0, T \Gamma_0) \quad \text{and} \quad  H^s_\beta(\Gamma_0; T^\perp \Gamma_0), H^s_{00,\beta} (\Gamma_0, T^\perp  \Gamma_0) . 
\]
Now choose $\beta > 0$ sufficiently small. A function $u \in H^1_{\beta}(\Omega;\CC^3)$ will be called outgoing if it satisfies a given asymptotic behaviour on the cylindrical ends. More precisely there exist functions $U_{j,+} \in   H^{2}_\beta(\Omega;\CC^3), j=1,\ldots, n$ such that 
\begin{equation*}
    \begin{aligned}
        (L^{(i)}  - k^2) U_{j,+} &= 0 , \text{ in } \Omega \setminus \tilde K \\
        U_{j,+} &= 0 , \text{ on } \partial \Omega ,
    \end{aligned}
\end{equation*}
for a suitable compact set $\tilde K \subseteq \overline{\Omega}$ and $i=1,2$. By definition we assume for an outgoing function $H^1_{\beta}(\Omega;\CC^3)$ to have the following asymptotic behaviour 
\begin{align}\label{eq:bvp_radiation}
	u  -  \sum_{j=1}^{n} c_j U_{j,+} \in H^1_{-\beta} (\Omega;\CC^3) 
\end{align}
for coefficients $c_j \in \CC$, $j=1,\ldots, n$. A similar assertion holds true for incoming functions, where we will  replace $U_{j,+}$ by incoming waves $U_{j,-}$. The precise definition is given in the appendix. 

In what follows we assume that the boundary value problem \eqref{eq:bvp_Poisson_cyl} with $f=0$ and $g=0$ has a unique outgoing solution and a unique incoming solution, namely $u=0$. Then we obtain the following result:
\begin{lemma}\label{lemma:solvability_bvp}
 	Let $f \in H^{1}_{\beta}(\Omega;\CC^3)^*$ and $g  \in H^{1/2}_{-\beta} (\Gamma_1;T^\perp \Gamma_1) \oplus H^{1/2}_{00,-\beta}(\Gamma_1;T\Gamma_1)$. Then the boundary value problem \eqref{eq:bvp_Poisson_cyl} has a unique outgoing and a unique incoming solution in $H^1_{\beta}(\Omega;\CC^3)$. 
\end{lemma}
The proof follows from the above assumption together with ideas in the appendix. Here we note that the boundary value problem will always be considered in its variational form.
In particular, for $g = (g_n , g_\tau)$ the solution $u$ of \eqref{eq:bvp_Poisson_cyl} shall satisfy  
\begin{align}
       \left\{ \begin{aligned} u_n &= g_n \text{ on } \Gamma_1 \\
        u_\tau &= g_\tau  \text{ on } \Gamma_1 \\
        u_\tau &= 0 \text{ on } \Gamma_0 , 
        \end{aligned} \right. 
\end{align}
as well as 
\begin{align}\label{eq:variation-cyl}
    \int_\Omega \scal{\sigma^{(i)} (u)}{\varepsilon (v) }_{\CC^3 \times \CC^3} \; \dd x  &= k^2 \int_{\Omega} u\;\bar{v} \; \dd x + \int_{\Omega} f \;  \bar{v} \; \dd x  ,
\end{align}
for every $v \in H^1_{-\beta}(\Omega;\CC^3)$ such that $v=0$ on $\Gamma_1$ and  $v_\tau = 0$ on $\Gamma_0$. 

Next we introduce the corresponding \DtN operator.
To this end let $g \in  H^{1/2}_{-\beta} (\Gamma_1;T^\perp \Gamma_1) \oplus H^{1/2}_{00,-\beta}(\Gamma_1;T\Gamma_1) $ and  $f =0$. Let $u \in H^{1}_{\beta}(\Omega;\CC^3) $ be the unique outgoing solution of the associated boundary value problem of
\eqref{eq:bvp_Poisson_cyl}.
The weak normal derivative of $u$, $B^{(i)} u \in H_{\beta}^{-1/2}(\partial \Omega;\CC^3)$, is again defined by Green's formula \eqref{eq:Greens_formula}, i.e. we have for all $v \in H^1_{-\beta}(\Gamma;\CC^3)$ that 
\begin{align*}
    \scal{B^{(i)} u}{v|_{\Gamma}} = \int_\Omega \scal{\sigma_i (u)}{\varepsilon (v) }_{\CC^3 \times \CC^3} \; \dd x - k^2 \int_\Omega u \overline{v} \; \dd x.
\end{align*}
We define as above 
\begin{align*}
    \Lambda^{(i)}    g  \coloneqq B^{(i)} u|_{\Gamma_1}
\end{align*}
and obtain an operator
\begin{align*}
    \Lambda^{(i)} : H^{1/2}_{-\beta} (\Gamma_1;T^\perp \Gamma_1) \oplus H^{1/2}_{00,-\beta}(\Gamma_1;T\Gamma_1)  \to H^{-1/2}_{00, \beta} (\Gamma_1;T^\perp\Gamma_1) \oplus H^{-1/2}_\beta(\Gamma_1;T\Gamma_1) . 
\end{align*}
As a next step we consider again the quadratic form
\begin{align*}
    H_\Omega (u_1 , u_2) :=  \int_\Omega \{ \lambda_2 - \lambda_1 \} \ddiv (u_1) \overline{\ddiv(u_2)} 
	+ 2\{\mu_2 - \mu_1\}  \scal{\varepsilon (u_1)}{\varepsilon(u_2)}_{\CC^3 \times \CC^3} \; \dd x.
\end{align*}
Note that the $\mu_i$ and $\lambda_i$ differ only on the compact set $K$, and thus, we have 
\begin{align*}
    H_\Omega (u_1 , u_2) &= H_K(u_1|_K, u_2|_K) \\
    &= \int_K \{ \lambda_2 - \lambda_1 \} \ddiv (u_1) \overline{\ddiv(u_2)} + 2\{\mu_2 - \mu_1\}  \scal{\varepsilon (u_1)}{\varepsilon(u_2)}_{\CC^3 \times \CC^3} \; \dd x  .
\end{align*}
The following theorem permits us to apply directly the results of the previous sections. Then Theorem \ref{thm:main1} and Theorem~\ref{thm:main2} for the infinite cylinder follows from the case of bounded domains.
Furthermore, Theorem~\ref{thm:inf-cyl} follows from \cite{EsRa}*{Theorem 1}. 
\begin{theorem}\label{th:h-K_cyl}
        If the \DtN operators are equal, $\Lambda^{(1)} = \Lambda^{(2)}$, then we have $H_K(u_1, u_2) = 0$ for solutions $u_i \in H^1(K;\CC^3)$ of 
        \begin{equation}\label{eq:L-k-K} \left\{ 
            \begin{aligned}
                    (L^{(i)} -k^2) u_i  &= 0  \text{ in } K \\
                    C^{(i)} u_i &= 0  \text{ on } \Gamma_0 . 
            \end{aligned} \right. 
        \end{equation}
\end{theorem}
The remaining part of the section is devoted to the proof of Theorem \ref{th:h-K_cyl}. 
Its proof is mainly based on a corresponding Runge approximation theorem.
\begin{theorem}\label{th:Runge}
    Let $u \in H^1(K;\CC^3)$ be solution of \eqref{eq:L-k-K} for $i = 1$.
	Then for every $\varepsilon >0$ there exists an outgoing solution $u_\varepsilon \in H^{1}_{\beta}(\Omega;\CC^3)$ of
    \begin{equation}\label{eq:bvp_Poisson_Runge}
        \left\{ 
        \begin{aligned}
            (L^{(1)} -k^2) u_\varepsilon  &= 0  \quad\text{ in } \Omega \\
            C^{(1)} u_\varepsilon &= 0  \quad\text{ on } \Gamma_0,
        \end{aligned} \right. 
    \end{equation}
    such that
	\[ \| u - u_\varepsilon\|_{H^1(K;\CC^3)} < \varepsilon  . \]
\end{theorem}
\begin{proof}
The argument is similar to the one in \cite{LiUhl}*{Lemma 3.3}.
Let $f \in H^1 (K;\CC^3)^*$ be an arbitrary functional with $f (u|_K) = 0$ for all outgoing $u \in H^{1}_{\beta}(\Omega;\CC^3)$ that satisfy \eqref{eq:bvp_Poisson_Runge}.
The assertion follows from the Hahn-Banach theorem, if we show that $f(v)=0$ for all $v \in H^1(K;\CC^3)$ such that $(L^{(1)} - k^2) v = 0$ in $\Omega$ and $C^{(1)} v =0$ on $\Gamma_0$.

We extend $f$ to a functional $  \tilde f  \in H^1_{\beta}(\Omega;\CC^3)^*$
by putting $\tilde f (u)  = f(u|_{K})$. 
Lemma~\ref{lemma:solvability_bvp} implies that there exists an incoming solution $v_0 \in H^{1}_{\beta}(\Omega;\CC^3)$ such that 
\begin{equation*} \left\{ 
        \begin{aligned}
		(L^{(1)} -  k^2) v_0  &= \tilde f   \text{ in } \Omega \\
		v_0 &= 0 \text{ on } \Gamma_1 \\
		C^{(1)} v_0 &= 0 \text{ on } \Gamma_0 .
		\end{aligned} \right. 
\end{equation*}
Note that $\tilde f = 0$ on $\Omega \setminus K$, and thus, from local elliptic regularity theorems we obtain $v_0 \in C^\infty(\overline{\Omega} \setminus K;\CC^3)$.

We show that $v_0 = 0$ on $\overline\Omega \setminus K$. For this let $g \in C_c^\infty(\partial\Omega;\CC^3)$ and consider the outgoing function $w_g \in H^1_{\beta}(\Omega;\CC^3)$, which satisfies
\begin{equation*}
        \left\{ 
        \begin{aligned}
        (L^{(1)} -  k^2) w_g  &= 0   \text{ in } \Omega \\
        w_g &= g \text{ on } \Gamma_1 \\
		C^{(1)} w_g &= 0 \text{ on } \Gamma_0 .
		\end{aligned} \right. 
\end{equation*}
If we assume that $w_g \in H^1_{-\beta}(\Omega;\CC^3)$, then the variational formulation \eqref{eq:variation-cyl} implies
\begin{align*}
    0 &= \int_{\Omega} \scal{\sigma(w_g)}{\varepsilon(v_0)} \; \dd x - k^2 \int_\Omega w_g \overline{v_0} \; \dd x. 
\end{align*}
However taking into account the choice of the radiation condition (cf.\ the appendix) the equality holds true for all $w_g$. Likewise we obtain
\begin{align*}
    0 &= \int_{\Omega} \scal{\sigma(v_0)}{\varepsilon(w_g)} \; \dd x - k^2 \int_\Omega v_0 \overline{w_g} \; \dd x\\
    &= \tilde f(w_g) + \scal{B^{(1)} v_0}{w_g|_{\partial \Omega}}_{\partial \Omega}\\
    &= \scal{B^{(1)} v_0|_{\Gamma_1} }{g|_{\Gamma_1}} .
\end{align*}
Since $g$ was chosen  arbitrarily  we have $B^{(1)} v_0 = 0$ on $\Gamma_1$.
Hence, as $(L^{(1)} -  k^2) v_0 = 0$ on $\Omega \setminus K$ we have also  $v_0 = 0$ on $\overline\Omega \setminus K$ by unique continuation, cf. \cite{Angetal}.
As a particular consequence we obtain $v_0 = 0$ on $\partial K  \setminus \Gamma_0$. 

Now let $u \in H^1(K;\CC^3)$ such that $(L^{(1)} - k^2) u = 0$ and $C^{(1)} u  =0$. Then the variational formulation applied to $u$ gives us that 
\begin{align*}
    0 &=  \int_K \scal{\sigma(u)}{\varepsilon(v_0)} \; \dd x - k^2 \int_K v_0 \overline{u} \; \dd x.
\end{align*}
As a next step we choose a sequence $\phi_n \in C_c^\infty(\overline{\Omega};\CC^3)$ with 
$\phi_n|_K \to u \in H^1(K;\CC^3)$ and $\phi_n = 0$ on $\Gamma_0$. Then the previous considerations imply
\begin{align*}
    0 &= \int_K \scal{\sigma(v_0)}{\varepsilon(u)} \; \dd x - k^2 \int_K v_0 \overline{u} \; \dd x\\
    &= \lim_{n \to \infty} \int_\Omega \scal{\sigma(v_0)}{\varepsilon(\phi_n)}  \; \dd x - k^2 \int_\Omega v_0 \overline{\phi_n} \; \dd x\\
    &= \lim_{n \to \infty} \tilde f (\phi_n) + \scal{B^{(1)} v_0|_{\Gamma_1}}{\phi_n|_{\Gamma_1}} \\
    &= \lim_{n \to \infty} \tilde f(\phi_n) = f(u) . 
\end{align*}
\end{proof}
Now, we can prove the main result of this section:
\begin{proof}[Proof of Theorem~\ref{th:h-K_cyl}]
Using Theorem~\ref{th:Runge}, it suffices to prove the following:
if $u_2$ is a solution of \eqref{eq:L-k-K} and 
$v_1 \in H^1(\Omega;\CC^3)$ an outgoing solution of \eqref{eq:bvp_Poisson_Runge},
then we have that $H_K(v_1|_K, u_2) = 0$.

We choose  $v \in H^1(\Omega;\CC^3)$ outgoing such that
\begin{equation*} \left\{ 
    \begin{aligned}
        (L^{(2)} -k^2) v  &= 0  &\text{in }& \Omega \\
        v &= v_1|_{\Gamma_1} &\text{on }& \Gamma_1 \\
        C^{(2)} v &= 0  &\text{on }& \Gamma_0.
    \end{aligned} \right. 
\end{equation*}
Then we have $B^{(1)} v_1|_{\Gamma_1}  = B^{(2)} v|_{\Gamma_1}$. Since $L^{(1)} = L^{(2)} = L $ on $\Omega \setminus K$ and $B^{(1)} = B^{(2)} = B$ on $\partial \Omega \setminus K$ we have for  $w \coloneqq v_1 - v$ that 
\[
    (L - k^2)w  = 0 \text{ on } \Omega \setminus K, \quad w|_{\partial \Omega \setminus \overline{K}} = 0 , \quad 
	B w|_{ \partial \Omega \setminus \overline{ K}} = 0 .
\]
The unique continuation theorem assures that $w=0$ on $\Omega \setminus K$, cf. \cite{Angetal}. 
In particular we have $w|_{\partial K \setminus \Gamma_0} =0$ and $w_{\tau} = 0$ on $\Gamma_0$ and we obtain with $\supp u_2 \subset K$ as before
\begin{align*}
    0 &= \int_{K} \scal{\sigma_2(w)}{\varepsilon(u_2)}_{\CC^3 \times \CC^3} \; \dd x - k^2 \int_K w\overline{u_2} \; \dd x\\
	&= \int_K \scal{\sigma_2(v_1)}{\varepsilon(u_2)}_{\CC^3 \times \CC^3} \; \dd x - k^2 \int_K v_1 \overline{u_2} \; \dd x - \scal{B^{(2)} v}{u_2}_{\partial K}.
\end{align*}
We note that $B^{(1)} v_1 = B^{(2)} v$ on $\partial \Omega \setminus \Gamma_0$ and due to local regularity theorems
we may assume that $w \in H^2$ in a neighbourhood of $\partial K \cap \Omega$ such that $B^{(1)} v_1 = B^{(2)} v$ on $\partial K \cap \Omega$.
This implies
\begin{align*}
	\int_K \scal{\sigma_2(v_1)}{\varepsilon(u_2)}_{\CC^3 \times \CC^3} \; \dd x &= \scal{B^{(2)} v}{u_2}_{\partial K} + k^2 \int_K v_1 \overline{u_2} \; \dd x\\
	&=  \scal{B^{(1)} v_1}{u_2}_{\partial K} + k^2 \int_K v_1 \overline{u_2} \; \dd x\\
	&=  \int_K \scal{\sigma_1(v_1)}{\varepsilon(u_2)}_{\CC^3 \times \CC^3} \; \dd x.
\end{align*}
In particular we have 
\[
    H_K(v_1|_K , u_2 ) = 0.
\]
\end{proof}

\appendix
\section{Radiation Conditions}
\label{sec:appendix}
For the sake of simplicity we assume that $\Gamma_0=\varnothing$. The general case follows likewise. In order to define the notion of incoming and outgoing functions, we follow the approach in \cites{KamNaz,NazPlam}. To this end we consider the operator $L = L(x_2,x_3, \partial_{x_1}, \partial_{x_2}, \partial_{x_3})$ corresponding to the Lam\'e coefficients at infinity, i.e. $\lambda$ and $\mu$ are given as in \eqref{eq:lame-infty}. For $k \in \RR$ we are concerned with  the following family of boundary value problems 
\[ \mathfrak A(\xi) \colon  H^2(G;\CC^3) \to \begin{array}{c} L^2(G;\CC^3) \\ \oplus \\ \CC^3 \end{array} , \qquad \mathfrak A(\xi) u = \begin{pmatrix} (L(x_2,x_3, i \xi , \partial_{x_2} , \partial_{x_3})  - k^2 )u \\[3pt] u|_{\partial G}  \end{pmatrix} . \]
It is well-known that for every $\xi \in \CC$ the operator  $\mathfrak A(\xi)$ is a Fredholm operator with vanishing index. Indeed, for the last assertion we observe that for real $\xi \in \RR$ the adjoint boundary value problem coincides with the original one.  Moreover, $\mathfrak A(\xi)$ is invertible for sufficiently large $\xi \in \RR$, and thus, it is invertible for all $\xi \in \CC$ with the possible exception of a discrete subset, cf. \cite{GohbergSigal} We put 
\begin{align*}
	\Xi := \{ \xi \in \CC : \ker \mathfrak A(\xi) \neq \{ 0 \} \} . 
\end{align*}
The elements of $\Xi$ are called characteristic values of the operator pencil $\mathfrak A$. We consider the boundary value problem
\begin{align}\label{eq:bvp_appendix}
	\left\{ \begin{aligned}
	(L^{(i)} -k^2) u  &= f  \text{ in } \Omega \\
	u &= g  \text{ on } \partial \Omega ,
\end{aligned} \right. 
\end{align}
where now $f \in H^1_{-\beta}(\Omega;\CC^3)^*$ and $g \in H_\beta^{1/2}(\partial \Omega ;\CC^3)$. The Fredholm property reads as follows.
\begin{lemma}
Let $\beta \in \RR$ such that $\Xi \cap (\RR \pm i \beta) = \varnothing$. Then either one of the following assertions hold true: 
\begin{enumerate}
	\item The homogeneous problem ($f=0$ and $g=0$) has the  unique solution $u=0$. Then \eqref{eq:bvp_appendix} has a unique solution $u \in H^1_\beta(\Omega;\CC^3)$.
 	\item There exists $n$ solutions $U_j \in H^1_\beta(\Omega;\CC^3)$ of the homogeneous problems. Then \eqref{eq:bvp_appendix}  is solvable if and only if 
 	\begin{align*}
        \scal{U_j}{f} + \scal{B^{(i)} U_j}{g}_{\partial \Omega}  = 0 , \qquad \text{for }   j =1, \ldots, n .  
 	\end{align*}
	\end{enumerate}
\end{lemma}
The proof follows as in \cite{KoMaRo}*{Corollary 3.4.2}. For the case of the mixed problem this has to be combined with the ideas in the proof of \cite{McLean}*{Theorem 4.10}. 

As a next step we want to consider the asymptotic behaviour of solutions of the boundary value problem \eqref{eq:bvp_appendix}.
For $\xi_0 \in \Xi$, we call
$v_0, \dotsc, v_\ell \in H^2(G) \setminus \{0\}$  a \emph{Jordan chain} of length $\ell+1$ (associated with $v_0$) for $\mathfrak A$, if
\begin{align*}
	0 = \sum_{q=0}^{j} \frac{1}{q!} \; \frac{\dd^q}{\dd \xi^q} \mathfrak A (\xi)\big|_{\xi=\xi_0} v_{j-q} , \qquad \text{for all } 
	j = 0, \dotsc , \ell . 
\end{align*}
Note that in particular we have that $v_0 \in \ker \mathfrak A(\xi_0)$.

If $v_0, \ldots, v_\ell$ is a Jordan-chain, then we define the  functions
\begin{equation}\label{eq:radiating_functions}
	V_j (x_1,x_2,x_3) = e^{i \xi_0 x_1} \sum_{q=0}^j \frac{(it)^q}{q !  } v_{j-q} (x_2,x_3) , \qquad j = 0, \ldots , \ell . 
\end{equation}

These functions satisfy the equation  $(L  - k^2) V_j  = 0$ in $\RR \times G$ and $V_j|_{\RR \times \partial G} = 0$. In what follows we assume that $\beta >0$ is chosen such that 
\[  \Xi \cap \{ \xi \in \CC : |\Im(\xi)| < \beta \} \subseteq \RR .  \]

We note that if  $\Xi \cap \RR = \varnothing$, then we do not need to impose any radiation conditions. 
Assume that $\Xi \cap \RR \neq \varnothing$ and denote by $N$ the total multiplicity of the characteristic values in $\Xi \cap \RR$. Note that $N$ is even. As before we obtain functions  $V_j$, $j=1, \ldots,N$.
Let $\chi_+ \in C^\infty(\RR)$ be chosen such that $\chi_+(t) = 0$ for $t \le R$ and $\chi_+(t) =1 $ for $t \ge R+1$, where $R > 0$ is sufficiently large. We put $\chi_- (x_1) \coloneqq \chi_+ (-x_1)$. Then we have the following result:
\begin{lemma}[cf. \cite{NazPlam}]\label{lemma:expansion_radiating_waves}
	Let $f \in H^1_{\beta} (\Omega;\CC^3)^*$ and $g \in H^{1/2}_{-\beta}(\partial \Omega;\CC^3)$ and let $u \in H^{1,\beta}(\Omega;	\CC^3)$ be a solution of  \eqref{eq:bvp_appendix}.
	Then there exists $\alpha_1, \ldots, \alpha_{N}, \beta_1, \ldots, \beta_N \in \CC$ such that 
	\begin{align*}
		u -   \chi_+  \sum_{j=1}^{N} \alpha_j  V_j  +   \chi_- \sum_{j=1}^{N} \beta _j   V_j  \in H^1_{-\beta}(\Omega;\CC^3). 
	\end{align*}
\end{lemma}
This leads us to consider the space
\begin{align*}
	\mathbb D &\coloneqq \sspan \{ \chi_+ V_j \colon j=1, \ldots, N \} + \sspan \{ \chi_- V_j \colon j=1, \ldots, N \} .
\end{align*}
Let
\begin{align*}
	q(u,v) \coloneqq \int_{\Omega} L_i u \overline{v} \; \dd x -  \int_{\Omega} u  \overline{L_i v} \; \dd x + \int_{\partial \Omega} B_i u \overline{v} \; \dd s - \int_{\partial \Omega} u \overline{ B_i v} \; \dd s , \quad  u,v \in \mathbb D . 
\end{align*}
Then we may choose a basis  $U_{1,+} \ldots, U_{N,+}$, $U_{1,-}, \ldots U_{N,- }$ of $\mathbb D$ such that 
\begin{align*}
	q(U_{j,+} , U_{j,+}) =  i\,\delta_{jk}  , \qquad q(U_{j,-} , U_{k,-}) = - i\,\delta_{jk} , 
	\qquad q(U_{j,+} , U_{k,-}) = 0 . 
\end{align*}
Then a function $u \in H^1_{\beta} (\Omega;\CC^3) $ will be  called outgoing if there are coefficients $c_j \in \CC$ such that
\begin{align}
	u  -  \sum_{j=1}^{n} c_j U_{j,+} \in H^1_{-\beta} (\Omega;\CC^3) 
\end{align}
Analogously, $u$ is called incoming if $u  -  \sum_{j=1}^{n} c_j U_{j,-} \in H^1_{-\beta} (\Omega;\CC^3)$. Then Lemma  \ref{lemma:solvability_bvp} will be a consequence of the Fredholm property in \cite{KamNaz}*{Theorem 2.2}. 

\end{document}